\documentclass[12pt]{article}

\usepackage{amsmath}
\usepackage{amsfonts}
\usepackage{latexsym}
\usepackage{amssymb}
\usepackage{curves}
\usepackage{graphicx}
\usepackage{wasysym}
\usepackage{MnSymbol}
\usepackage{amsthm,amscd}

\newtheorem{thm}{Theorem}[section]
\newtheorem{pro}[thm]{Proposition}
\newtheorem{cor}[thm]{Corollary}
\newtheorem{lem}[thm]{Lemma}
\newtheorem{conj}[thm]{Conjecture}

\newtheorem{rem}[thm]{Remark}

\newtheorem{ques}[thm]{Question}

\newcommand{\comp}{\, {}_\circ \,}

\newcommand{\noin}{\noindent}

\newcommand{\FF}{{\mathbb{F}}}

\newcommand{\openbin}{\left( \!\! \begin{array}{c}}
\newcommand{\closebin}{\end{array} \!\! \right)}
\newcommand{\openvec}{\left[ \!\! \begin{array}{c}}
\newcommand{\closevec}{\end{array} \!\! \right]}
\newcommand{\openmat}{\left[ \!\! \begin{array}{cc}}
\newcommand{\closemat}{\end{array} \!\! \right]}
\newcommand{\opentri}{\left[ \!\! \begin{array}{ccc}}
\newcommand{\closetri}{\end{array} \!\! \right]}
\newcommand{\openquad}{\left[ \!\! \begin{array}{cccc}}
\newcommand{\closequad}{\end{array} \!\! \right]}
\newcommand{\openquin}{\left[ \!\! \begin{array}{ccccc}}
\newcommand{\closequin}{\end{array} \!\! \right]}
\newcommand{\opensex}{\left[ \!\! \begin{array}{cccccc}}
\newcommand{\closesex}{\end{array} \!\! \right]}

\newcommand{\dash}{\mbox{\rm{--}}}

\newcommand{\End}{\mbox{\rm{End}}}

\newcommand{\Ind}{\mbox{\rm{Ind}}}

\newcommand{\Res}{\mbox{\rm{Res}}}

\newcommand{\br}{\mbox{\rm{br}}}

\newcommand{\cod}{\mbox{\rm{cod}}}

\newcommand{\dom}{\mbox{\rm{dom}}}

\newcommand{\id}{{\rm{id}}}

\newcommand{\tr}{\mbox{\rm{tr}}}

\newcommand{\oo}{\overline}
\newcommand{\uu}{\underline}

\newcommand{\cF}{{\cal F}}

\newcommand{\cI}{{\cal I}}

\newcommand{\cO}{{\cal O}}

\newcommand{\cU}{{\cal U}}
\newcommand{\cV}{{\cal V}}

\begin{document}


\title{Conjectural invariance with respect to \\
the fusion system of an almost-source algebra}

\author{\large Laurence
Barker\footnote{e-mail: barker@fen.bilkent.edu.tr.
Some of this work was done while this author was
on sabbatical leave, visiting the Department of
Mathematics at City, University of London.}
\hspace{1in} Matthew Gelvin\footnote{e-mail:
mgelvin@gmail.com.} \\
\mbox{} \\
Department of Mathematics \\
Bilkent University \\
06800 Bilkent, Ankara, Turkey \\
\mbox{}}

\maketitle

\small
\begin{abstract}
\noin We show that, given an almost-source
algebra $A$ of a $p$-block of a finite group
$G$, then the unit group of $A$ contains a basis
stabilized by the left and right multiplicative
action of the defect group if and only if, in a
sense to be made precise, certain relative
multiplicities of local pointed groups are
invariant with respect to the fusion system.
We also show that, when $G$ is $p$-solvable,
those two equivalent conditions hold for some
almost-source algebra of the given $p$-block.
One motive lies in the fact that, by a
theorem of Linckelmann, if the two equivalent
conditions hold for $A$, then any stable
basis for $A$ is semicharacteristic for the
fusion system.

\smallskip
\noin {\it 2020 Mathematics Subject Classification:}
20C20.

\smallskip
\noin {\it Keywords:} characteristic biset, source
algebra, fusion system, $p$-solvable group

\end{abstract}

\section{Introduction}
\label{1}

In $p$-local representation theory of
finite groups, when dealing with an
object $P$ of a fusion system, we
often consider attributes of $P$ that
are not invariants of $P$. That is, we
often consider conditions on $P$,
constructions determined by $P$,
numbers associated with $P$ that
are not isomorphism invariants with
respect to the fusion system. The
two conjectures below imply that
certain attributes of $P$ are, despite
appearances, invariants after all.

To better explain what we mean, we
must first set the scene. Throughout,
we let $\cO$ be a complete local
Noetherian commutative unital ring with
an algebraically closed residue field
$\FF$ of prime characteristic $p$. The
hypothesis on $\cO$ implies that either
$\cO = \FF$ or else $\cO$ is a
complete discrete valuation ring. We
understand any $\cO$-module, and in
particular any algebra over $\cO$, to be
finitely generated and free over $\cO$
or over $\FF$. We understand a
{\bf basis} for an $\cO$-module to
be an $\cO$-basis or an $\FF$-basis.
A basis $\Omega$ for an algebra over
$\cO$ is said to be {\bf unital} provided
every element of $\cO$ is a unit.

Throughout, we let $G$ be a finite group.
We deem $G$-algebras to be over $\cO$.
We presume familiarity with the theory of
$G$-algebras, Brauer pairs, almost-source
algebras, fusion systems, as discussed in
Linckelmann \cite[Chapters 5, 6, 8]{Lin18}.
We shall freely make use of much notation
and terminology from there. However,
adopting two definitions in Craven
\cite[4.1, 4.11]{Cra11}, we understand the
term {\it fusion system} in the broad sense,
as distinct from {\it saturated fusion system}.

For a finite $G {\times} G$-set $\Gamma$,
the image of $c \in \Gamma$, under the
action of $(f, g) \in G {\times} G$, can be
written as $(f, g) c$ or as $f c g^{-1}$.
When the latter notation prevails, we call
$\Gamma$ a {\bf $(G, G)$-biset}. Regard
$G$ as a $(G, G)$-biset by left and right
multiplication. Extending $\cO$-linearly,
$\cO G$ becomes a permutation
$\cO(G {\times} G)$-module.

The scenario of main concern in this paper
is as follows. Let $b$ be a block of $\cO G$.
Let $D$ be a defect group of the block algebra
$\cO G b$. The notion of an almost-source
algebra and its associated fusion system will
be recalled in Section \ref{7}. Let $A$ be an
almost-source $D$-algebra of $\cO G b$.
Let $\cF$ be the fusion system on $D$
associated with $A$. By Green's
Indecomposability Criterion, $A$ is a
permutation $\cO(D {\times} D)$-submodule
of $\cO G$. Let $\Omega$ be a
$D {\times} D$-stable basis for $A$. By the
Krull--Schmidt Theorem, the $(D, D)$-biset
$\Omega$ is well-defined up to isomorphism.

A result of Linckelmann, Theorem
\ref{7.2} below, describes how $\cF$
is determined by $\Omega$. One
clause of the description involves a
condition that is not $\cF$-isomorphism
invariant. Theorem \ref{7.5} asserts
that the condition can be omitted when
$b$ is the principal block of $\cO G$.
Most of this paper, though, is concerned
with some apparently stronger
properties of $A$ and $\Omega$.

The notion of an $\cF$-semicharacteristic
biset is defined in Gelvin--Reeh \cite{GR15}.
We shall recall the definition in Section
\ref{6}. The next result is presented in
\cite[8.7.11]{Lin18} for source algebras,
but the proof carries over immediately to
almost-source algebras.

\begin{thm}
\label{1.1}
{\rm (Linckelmann.)}
Let $b$ be a block of $\cO G$. Let $D$
be a defect group of the block algebra
$\cO G b$. Let $A$ be an almost-source
$D$-algebra of $\cO G b$. Write $\cF$
for the fusion system on $D$ associated
with $A$. If $A$ has a unital
$D {\times} D$-stable basis $\Omega$,
then the $(D, D)$-biset $\Omega$ is
$\cF$-semicharacteristic. If, furthermore,
$A$ is a source algebra, then $\Omega$
is $\cF$-characteristic.
\end{thm}

The weaker of the two conjectures below
asserts that some almost-source algebra
of $\cO G b$ does have a stable basis.
The stronger asserts that the source
algebras of $\cO G b$ have stable bases.
The next theorem gives two necessary
and sufficient criteria for the stability
hypothesis.

We understand an {\bf idempotent
decomposition} in a ring to be a sum
of finitely many mutually orthogonal
idempotents. When the summands are
primitive idempotents, we call the sum a
{\bf primitive idempotent decomposition}.
Given a pointed group $U_\mu$ on any
$G$-algebra $A$, we define the
{\bf multiplicity} of $U_\mu$, denoted
$m_A(U_\mu)$, to be the number of
elements of $\mu$ appearing in a
primitive idempotent decomposition
of the unity element of the $U$-fixed
subalgebra $A^U$. Given a pointed group
$T_\tau$ on $A$ with $T \leq U$, we
define the {\bf relative multiplicity} of
$T_\tau$ in $U_\mu$, denoted
$m_A(T_\tau, U_\mu)$, to be the
number of elements of $\tau$ that appear
in a primitive idempotent decomposition
of $i$ in the algebra $A^T$, where $i$ is
any element of $\mu$. Observe that, given
a pointed group $U_\mu$ on $A$ and
$U \leq V \leq G$ then, by considering
a primitive idempotent decomposition in
$A^V$ and refining to a primitive
idempotent decomposition in $A^U$,
we obtain
$$m_A(U_\mu) = \sum_\nu
  m_A(U_\mu, V_\nu) \, m_A(V_\nu)$$
where $\nu$ runs over the points of $V$
on $A$. Similarly, given pointed groups
$U_\mu$ and $W_\omega$ on $A$ and
$U \leq V \leq W$, then
$$m_A(U_\mu, W_\omega) =
  \sum_\nu m_A(U_\mu, V_\nu)
  \, m_A(V_\nu, W_\omega)$$
with $\nu$ running as before.

Now suppose $A$ is an interior
$G$-algebra. Let $U_\mu$ and
$V_\nu$ be pointed groups on
$A$. Suppose $U \cong V$ and let
$\phi : V \rightarrow U$ be a group
isomorphism. We call $\phi$ an
{\bf isofusion} $V_\nu \rightarrow U_\mu$
provided, in the sense of
Puig \cite[Section 2]{Pui86}, $\phi$ is an
``$A$-fusion'' $V_\nu \rightarrow U_\mu$.
In Section \ref{4}
we shall reviewing that notion, and
we shall also be noting some further
characterizations of isofusions.

\begin{thm}
\label{1.2}
Let $b$ be a block of $\cO G$ with defect
group $D$ and almost-source $D$-algebra
$A$. Write $\cF$ for the fusion system on
$D$ associated with $A$. Then the following
three conditions are equivalent.

\noin {\bf (a)} For every $\cF$-isomorphism
$\phi : Q \rightarrow P$, there is a bijective
correspondence between the local points
$\gamma$ of $P$ on $A$ and the local
points $\delta$ of $Q$ on $A$ whereby
$\gamma \leftrightarrow \delta$ if and only
if $\phi$ is an isofusion $Q_\delta \rightarrow
P_\gamma$. Furthermore,
when $\gamma \leftrightarrow \delta$, we
have $m_A(P_\gamma) = m_A(Q_\delta)$.

\noin {\bf (b)} For every $\cF$-isomorphism
$\phi : Q \rightarrow P$, there is a bijective
correspondence between the points $\alpha$
of $P$ on $A$ and the points $\beta$ of $Q$
on $A$ whereby $\alpha \leftrightarrow
\beta$ if and only if $\phi$ is an isofusion
$Q_\beta \rightarrow P_\alpha$. Furthermore,
when $\alpha \leftrightarrow \beta$, we
have $m_A(P_\alpha) = m_A(Q_\beta)$

\noin {\bf (c)} There exists a unital
$D {\times} D$-stable basis for $A$.
\end{thm}

We shall prove that theorem in Section
\ref{8}. When the equivalent conditions
(a), (b), (c) hold, we call $A$ a {\bf uniform
almost-source $D$-algebra} of $\cO G b$.

Condition (b) can be mnemonically
summarized as: multiplicities of pointed
groups are $\cF$-invariant, condition (a)
likewise. Let us explain another way of
interpreting condition (a) as an assertion
of $\cF$-invariance. Recall that, for each
$\cF$-centric subgroup $P$ of $D$, there
is a unique local point $\lambda_P^A$ of
$P$ on $A$. Given $\cF$-centric
subgroups $P$ and $Q$ of $D$ and an
$\cF$-isomorphism $\phi : Q \rightarrow P$,
then $\phi$ is an isofusion $Q_{\lambda_Q^A}
\rightarrow P_{\lambda_P^A}$. Condition (a),
confined to the centric objects of $\cF$,
says that the function $P \mapsto
m_A(P_{\lambda_P^A})$ is constant on
$\cF$-isomorphism classes.

\begin{conj}
\label{1.3} Let $b$ be a block of $\cO G$.
Then the block algebra $\cO G b$ has a
uniform almost-source algebra.
\end{conj}

Part of our evidence for the conjecture
is the following result, proved in
Section \ref{8}.

\begin{thm}
\label{1.4}
Suppose $G$ is $p$-solvable. Given a block
$b$ of $\cO G$, then the block algebra
$\cO G b$ has a uniform almost-source
algebra.
\end{thm}

For source algebras, the multiplicities
appearing in conditions (a) and (b) of
Theorem \ref{1.2} can be realized as
relative multiplicities. In Section \ref{8}
we shall explain how, in the case of a
source algebra of $\cO G b$,
Theorem \ref{1.2} specializes as follows.

\begin{thm}
\label{1.5}
Let $b$ be a block of $\cO G$ with defect
group $D$ and source $D$-algebra $A$.
Write $\cF$ for the fusion system on $D$
associated with $A$. Let $\lambda_D$
be the unique local point of $D$ on $\cO G$
such that $\lambda_D \cap A \neq
\emptyset$. Then the following three
conditions are equivalent:

\noin {\bf (a)} For every $\cF$-isomorphism
$\phi : Q \rightarrow P$, there is a bijective
correspondence between: the local points
$\gamma$ of $P$ on $\cO G$ satisfying
$P_\gamma \leq D_{\lambda_D}$; the local
points $\delta$ of $Q$ on $\cO G$
satisfying $Q_\delta \leq D_{\lambda_D}$.
The correspondence is such that $\gamma
\leftrightarrow \delta$ if and only if $\phi$
is an isofusion $Q_\delta \rightarrow
P_\gamma$. Furthermore, when $\gamma
\leftrightarrow \delta$, we have
$m_{\cO G}(P_\gamma, D_{\lambda_D})
= m_{\cO G}(Q_\delta, D_{\lambda_D})$.

\noin {\bf (b)} Upon removing the term
{\it local} from (a), the condition still holds.

\noin {\bf (c)} There exists a unital
$D {\times} D$-stable basis for $A$.
\end{thm}

By the uniqueness of source algebras up
to conjugation, if some source algebra of
$\cO G b$ is uniform, then every source
algebra of $\cO G b$ is uniform.

For any source $D$-algebra $A$, each
$\cF$-centric $P \leq D$ can be
associated with a relative multiplicity, as
follows. For a reason recalled in Section
\ref{8}, the $\cF$-centricity of $P$ implies
the existence of a unique local point
$\lambda_P$ of $P$ on $\cO G$ such that
$P_{\lambda_P} \leq D_{\lambda_D}$. We
define $m_A(P) = m_{\cO G}(P_{\lambda_P},
D_{\lambda_D})$. Uniformity implies that
the relative multiplicity $m_A(P)$ is
$\cF$-invariant. We mean to say, if $A$
is uniform, then $m_A(P)$ depends only
on the $\cF$-isomorphism class of $P$.

\begin{conj}
\label{1.6}
For any block $b$ of $\cO G$, the source
algebras of $\cO G b$ are uniform.
\end{conj}

In Section \ref{8}, we shall show that
Conjecture \ref{1.6} holds in the following
special case. Recall, a fusion system
$\cF$ on a finite $p$-group $D$ satisfies
$\cF = N_\cF(D)$ if and only if every
$\cF$-isomorphism is a restriction of
an $\cF$-automorphism of $D$.

\begin{pro}
\label{1.7}
In the notation of Theorem \ref{1.5},
if $\cF = N_\cF(D)$, then the
source $D$-algebras of $\cO G b$ are
uniform. In particular, the conclusion
holds whenever $D$ is abelian or
$D \unlhd G$.
\end{pro}

Let us outline the organization of the
paper. The next three sections may be
of use in applications quite different from
our main concerns here. Section \ref{2}
supplies a necessary and sufficient criterion
for an interior $G$-algebra to have a
$G {\times} G$-stable basis. In
Section \ref{3}, we show that, for any
$G$-algebra, distinct points of a
$p$-subgroup cannot share a defect
pointed subgroup. In Section \ref{4},
we present a new approach to fusions
in interior $G$-algebras. The remaining
four sections are focused on proofs of
the new results stated in this introduction.
Sections \ref{5} and \ref{6} establish the
crux of Theorem \ref{1.2} but, for simplicity,
in a more abstract setting.

A $(D, D)$-biset $\Omega$ is said to be
{\bf $\cF$-divisible} provided $D {\times} 1$
and $1 {\times} D$ act freely on $\Omega$
and, given any isomorphism $\phi$
between subgroups of $D$, then $\phi$
is an $\cF$-isomorphism if and only if
there exists $w \in \Omega$ satisfying
$\phi(y) w y^{-1} = w$ for all $y$ in the
domain of $\phi$. In Section 7, we shall
prove the following theorem.

\begin{thm}
\label{1.8}
Let $b$, $D$, $A$, $\cF$ be as in
Theorem \ref{1.1}. Let $\Omega$ be
a $D {\times} D$-stable basis for $A$.
If $\cF$ is $p$-constrained or $G$ is
$p$-solvable or $b$ is the principal
block of $\cO G$, then $\Omega$ is
$\cF$-divisible.
\end{thm}

In the final section, we shall complete
the proofs of the results 1.2, 1.4, 1.5, 1.7.

\section{Unital stable bases}
\label{2}

Our hypothesis on $\cO$ implies that
every algebra over $\cO$ has a unital
basis. More generally, for any finite
$p$-group $D$, any permutation
$D$-algebra has a unital $D$-stable
basis, indeed, a further generalization,
with a weaker condition on the coefficient
ring, can be found in Linckelmann
\cite[5.8.13]{Lin18}. The purpose of
this section is to give a criterion for
an interior $G$-algebra to have a
unital $G {\times} G$-stable basis.

As a reminder of a convention
announced in Section 1, let us repeat
that any $\cO$-module $X$ is deemed
to have a finite basis, we mean, a finite
subset $\Omega$ such that $\Omega$
is an $\cO$-basis when $X$ is free,
an $\FF$-basis when $X$ is annihilated
by $J(\cO)$. We understand the {\bf rank}
of $X$ to be the well-defined natural
number $|\Omega|$. We define $\uu{X}
= X / J(\cO)X$ as an $\FF$-module. We
write $\uu{w}$ and $\uu{\Omega}$ for
the images in $\uu{X}$ of an element
$w$ of $X$ and a subset $\Omega$
of $X$.

\begin{lem}
\label{2.1}
Given a free $\cO$-module $X$ and
$\Omega \subseteq X$ such that
the reduction map $X \rightarrow \uu{X}$
restricts to a bijection $\Omega
\rightarrow \uu{\Omega}$, then $\Omega$
is an $\cO$-basis for $X$ if and only
if $\uu{\Omega}$ is an $\FF$-basis
for $\uu{X}$.
\end{lem}

\begin{proof}
This easy exercise demands, at most,
just a hint: when $\uu{\Omega}$ is an
$\FF$-basis, express each element
of $\Omega$ as an $\cO$-linear
combination of elements of an
$\cO$-basis for $X$, then consider
the determinant of the square
matrix formed by the coefficients.
\end{proof}

\begin{lem}
\label{2.2}
Let $\Omega$ be a basis for an
$\cO$-module $X$. For each
$w \in \Omega$, let $v_w \in X$. Taking
$\lambda$ to run over the elements of
$\cO$ then, for all except at most
$|\Omega|$ values of the reduction
$\uu{\lambda} \in \FF$, the set
$\{ w + \lambda v_w : w \in \Omega\}$
is a basis for $X$.
\end{lem}

\begin{proof}
By the previous lemma, we may assume
that $X$ is annihilated by $J(\cO)$. Hence,
in fact, we may assume that $\cO = \FF$.
Let $w(\lambda) = w + \lambda v_w$. Let
$M(\lambda)$ be the matrix, with rows and
columns indexed by $\Omega$, such that
the $(w', w)$-entry $M_{w', w}(\lambda)$
of $M(\lambda)$ is given by $w(\lambda)
= \sum_{w'} M_{w', w}(\lambda) w'$. The
set $\{ w(\lambda) : w \in \Omega \}$ is
an $\FF$-basis for $M$ if and only if
$M(\lambda)$ is invertible. The function
$\lambda \mapsto \det(M(\lambda))$
is a polynomial function over $\FF$ with
degree at most $\Omega$ and with
nonzero constant term. Therefore,
$\det(M(\lambda)) = 0$ for at most
$|\Omega|$ values of $\lambda$.
\end{proof}

\begin{lem}
\label{2.3}
Let $A$ be an algebra over $\cO$, let
$n$ be the rank of $A$ and let $a \in A$
and $u \in A^\times$. Taking $\lambda$
to run over the elements of $\cO$ then,
for all except at most $n$ values
of the reduction $\uu{\lambda} \in \FF$,
we have $a + \lambda u \in A^\times$.
\end{lem}

\begin{proof}
Since units of $\uu{A}$ lift only to units
in $A$, we may assume that $\cO = \FF$.
Let $\rho : A \rightarrow \End_\cO(A)$ be
the regular representation. We have
$a + \lambda u \in A^\times$ if and
only if $\rho(a + \lambda u) \in
\End_\cO(A)^\times$, equivalently,
$\rho(- u^{-1} a) - \lambda . \id_A \in
\End_\cO(A)^\times$, in other words,
$\lambda$ is not an eigenvalue of
$\rho(- u^{-1} a)$.
\end{proof}

When discussing interior $G$-algebras, we
shall freely use notation and terminology
from Linckelmann \cite[Chapter 5]{Lin18},
but let us reiterate a few conventions. Let
$A$ be an interior $G$-algebra. For
$g \in G$ and $a \in A$, we write
$ga = \sigma_A(g)a$, similarly for $ag$,
where $\sigma_A : G \rightarrow A^\times$
is the structural homomorphism of $A$.
As usual, we regard $A$ as an
$\cO G$-module via $\sigma_A$. That is,
$g$ sends $a$ to ${}^g a = g a g^{-1}$.
We also regard $A$ as an
$\cO(G {\times} G)$-module, with
$(f, g) \in G {\times} G$ sending $a$ to
$f a g^{-1}$.

\begin{thm}
\label{2.4} Let $A$ be an interior
$G$-algebra. Then $A$ has a unital
$G {\times} G$-stable basis if and only
if $A$ has a $G {\times} G$-stable basis
$\Omega$ such that, for all
$w \in \Omega$, the stabilizer
$N_{G \times G}(w)$ fixes a unit of $A$.
\end{thm}

\begin{proof}
One direction is obvious. Conversely,
suppose there exists $\Omega$ as
specified. For $a \in A$, we write
$N(a) = N_{G \times G}(a)$. Letting
$w$ run over representatives of the
$G {\times} G$-orbits of $\Omega$, we
choose elements $u(w) \in A^\times
\cap A^{N(w)}$. Now letting $w$ run
over all the elements of $\Omega$, we
let $w \mapsto u(w)$ be the unique
function such that $u(f w g^{-1}) =
f u(w) g^{-1}$ for all $w$ and all
$f, g \in G$. We have $N(u(f w g^{-1}))
= {}^{(f, g)} N(u(w))$. So $N(w) = N(u(w))$
for all $w \in \Omega$.

Let $n = |\Omega|$. Lemma \ref{2.3}
implies that, letting $\lambda$ run
over the elements of $\cO$ then, for
each $w \in \Omega$, the element
$w_\lambda - w + \lambda u(w)$ is
a unit for all except at most $n$ values
of $\oo{\lambda}$. Lemma \ref{2.2}
implies that $A$ has basis
$\Omega_\lambda = \{ w_\lambda :
w \in \Omega \}$ for all except at
most $n$ values of $\oo{\lambda}$.
Therefore, $\Omega_\lambda$ is a
unital basis for all except at most
$n^{n+1}$ values of $\lambda$.
\end{proof}

\begin{cor}
\label{2.5}
Let $e$ be an idempotent of $Z(\cO G)$.
Let $S$ be a Sylow $p$-subgroup of $G$.
Then the algebra $\cO G e$ has a
unital $S {\times} S$-stable basis.
\end{cor}

\begin{proof}
Let $A = \cO G e$. We can regard $A$
as a direct summand of the permutation
$\cO(S {\times} S)$-module $\cO G$.
So $A$ has an $S {\times} S$-stable
basis $\Omega$, moreover, any such
$\Omega$ is isomorphic to an
$S {\times} S$-subset of the basis $G$
of $\cO G$. Given $w \in \Omega$,
then $N_{S \times S}(w) =
N_{S \times S}(g)$ for some $g \in G$.
On the other hand, viewing $A$ as an
algebra, the unit $ge \in A$ also has
stablizer $N_{S \times S}(ge) =
N_{S \times S}(g)$. The required
conclusion now follows by applying
the latest theorem to $A$ as an
interior $S$-algebra.
\end{proof}

\section{Green's Indecomposability
Criterion, revisited}
\label{3}

The little theorem in this section, and
the subsequent remark, will be needed
in Section \ref{5}.
Let $D$ be a finite $p$-group and $A$
a $D$-algebra. A theorem of Puig, in
Linckelmann \cite[5.12.20]{Lin18}, implies
that, given a point $\alpha$ of $D$ on
$A$ and a defect pointed subgroup
$P_\gamma$ of $D_\alpha$, then any
element of $\alpha$ can be expressed
in the form $\tr_P^D(i)$, where $i \in
\gamma$ and the idempotents ${}^g i$
of $A$ are mutually orthogonal as $gP$
runs over the left cosets of $P$ in $D$. In
\cite[5.12]{Lin18}, it is explained how the
theorem can be seen as a generalization
of Green's Indecomposability Criterion.

\begin{thm}
\label{3.1}
Given a $D$-algebra $A$, a local
pointed group $P_\gamma$ and points
$\alpha$ and $\beta$ of $D$ on $A$
such that $P_\gamma$ is a defect
pointed subgroup of $D_\alpha$ and
$D_\beta$, then $\alpha = \beta$.
\end{thm}

\begin{proof}
Let $a \in \alpha$ and $b \in \beta$.
By Puig's generalization of Green's
Indecomposability Theorem, there exist
$i, j \in \gamma$ such that
$a = \tr_P^D(i)$ and $b = \tr_P^D(j)$,
with ${}^g i . i = 0 = {}^g j . j$ for all
$g \in D - P$. Let $r \in (A^P)^\times$
such that $i = {}^r j$. Define
$u = \tr_P^D(irj)$ and $v = \tr_P^D(jri)$.
By direct calculation, $uv = i$ and
$vu = j$. So the idempotents $i$ and
$j$ of $A^D$ are associate, hence
conjugate.
\end{proof}

Proof of the following remark is a
straightforward application of
Mackey decomposition.

\begin{rem}
\label{3.2}
Given a $D$-algebra $A$, a point $\alpha$
of $D$ on $A$ and a defect pointed
subgroup $P_\gamma$ of $D_\alpha$,
then $m_A(P_\gamma, D_\alpha) =
|N_D(P_\gamma) : P|$.
\end{rem}

\section{Fusions in interior $G$-algebras}
\label{4}

We shall discuss Puig's notion of a fusion
between pointed groups. Some of the
criteria we shall give for the existence of
fusions do not appear to be widely known.
Throughout this section, we let $A$ be
an interior $G$-algebra.

Given a group isomorphism $\theta$, we
write $\cod(\theta)$ for the codomain,
$\dom(\theta)$ for the domain, and we
define $\Delta(\theta) = \{ (\theta(v), v)
: v \in \dom(\theta) \}$ as a subgroup of
$\cod(\theta) \times \dom(\theta)$. Note
that, given $U, V \leq G$ and a group
isomorphism $\phi : V \rightarrow U$, then
the $\Delta(\phi)$-fixed $\cO$-submodule
$A^\phi = A^{\Delta(\phi)}$ is the
$\cO$-submodule of elements $a \in A$
such that $\phi(v) a = av$ for all $v \in V$.
Condition (a) in the next lemma was
considered by Puig \cite[2.5, 2.12]{Pui86}.
He used it in his definition of a fusion, as he
called it, an ``$A$-fusion''.

\begin{lem}
\label{4.1}
Let $U_\mu$ and $V_\nu$ be pointed
groups on $A$. Let $\phi : V \rightarrow U$
be a group isomorphism. Choose
$i \in \mu$ and $j \in \nu$. Then the
following two conditions are equivalent:

\noin {\bf (a)} there exists $r \in A^\times$
such that $\phi(v) i = {}^r (vj)$,

\noin {\bf (b)} there exist $s \in i A^\phi j$
and $s' \in j A^{\phi^{-1}}i$ such that
$i = s s'$ and $j = s' s$.

Moreover, the equivalent conditions
(a) and (b) are independent of the
choices of $i$ and $j$.
\end{lem}

\begin{proof}
Assuming (a) and putting $s = irj$ and
$s' = j r^{-1} i$, we deduce (b). Conversely,
assume (b). Since associate idempotents of
a semiperfect ring are conjugate, there exists
$q \in A^\times$ such that $i = q j q^{-1}$.
Putting $r = s + (1 - i)q(1 - j)$ and
$r' = s' + (1 - j) q^{-1} (1 - i)$, then
$r r' = 1 = r' r$. So $r \in A^\times$. A
straightforward manipulation yields the
equality in (a). Confirmation of the rider
is routine.
\end{proof}

When conditions (a) and (b) hold, we
call $\phi$ an {\bf isofusion} $V_\nu
\rightarrow U_\mu$ and we call $(s, s')$
a {\bf $\phi$-witness} $j \mapsto i$.
Witnesses can be combined in the
following way. Let $U_\mu$, $V_\nu$,
$W_\omega$ be pointed groups on $A$.
Let $i \in \mu$, $j \in \nu$,
$k \in \omega$. Let $(s, s') :
j \mapsto i$ and $(t, t') :
k \mapsto j$ be witnesses for group
isomorphisms $\phi : V \rightarrow U$
and $\psi : W \rightarrow V$, respectively.
Then $(st, t's') : k \mapsto i$ is a
witness for $\phi \psi$. Also, $(s', s)$
is a witness for $\phi^{-1}$.

The next two lemmas are implicit in
\cite[2.5, 2.6, 2.7]{Pui86}. They can
be proved easily and routinely using
condition (a) in Lemma \ref{4.1}.
Alternatively, they can be proved very
quickly using the above comments
about witnesses.

\begin{lem}
\label{4.2}
{\rm (Puig.)}
Let $U, V \leq G$ and $\phi : V \rightarrow U$
a group isomorphism.

\noin {\bf (1)} For each point $\nu$ of $V$
on $A$, there is at most one point $\mu$
of $U$ on $A$ such that $\phi$ is an
isofusion $V_\nu \rightarrow U_\mu$.

\noin {\bf (2)} For each point $\mu$ of $U$
on $A$, there is at most one point $\nu$
of $V$ on $A$ such that $\phi$ is an
isofusion $V_\nu \rightarrow U_\mu$.

\noin {\bf (3)} Given points $\mu$ of $U$
and $\nu$ of $V$ on $A$ such that
$\phi$ is an isofusion $V_\nu \rightarrow
U_\mu$, then $\phi^{-1}$ is an
isofusion $U_\mu \rightarrow V_\nu$
\end{lem}

\begin{lem}
\label{4.3}
{\rm (Puig.)}
There is a groupoid whose objects are
the pointed groups on $A$ and whose
isomorphisms are the isofusions, the
composition being the usual composition
of group isomorphisms. In other words,
given a pointed group $U_\mu$ on $A$,
then the identity automorphism $\id_U$
is an isofusion $U_\mu \rightarrow U_\mu$,
moreover, the isofusions between the
pointed groups on $A$ are closed under
composition and inversion.
\end{lem}

We shall not be making in explicit study
of the groupoid in the latest lemma. For
that reason, we refrain from giving the
groupoid a symbol or name. We shall,
however, be making frequent use of the
lemma, without further signal.

There are two minor ways in which our
approach differs from that in Puig
\cite{Pui86}. The differences are of
shallow significance, but we comment
on them to avoid misunderstanding.
Puig first considered, more generally,
fusions between pointed groups on
$A$. That is, he considered composites
of inclusions and isofusions. He then
confined attention to the local pointed
groups, forming the ``local fusion
category'', \cite[2.15]{Pui86}.

As for the first difference, we never
have call, in this paper, to discuss
non-isomorphisms in the role of
fusions. Even when working with the
fusion system $\cF$ introduced in
Section \ref{1}, we shall be focusing
exclusively on the $\cF$-isomorphisms.
Treatment of other $\cF$-morphisms
would incur cost without gain.

As for the second, one reason for our
paying attention to the set of all points of
a $p$-subgroup, in the context of a given
$G$-algebra, is that such sets appear in
condition (b) of Theorem \ref{5.2} below.
Were we to omit (b) from the statement
of the theorem, our proof would still
crucially involve the consideration of
arbitrary points of $p$-subgroups.
The relevance of all those points
is not just a peculiarity of the proof.
Techniques for calculating relative
multiplicities lie outside the scope of
this paper but, as can readily be gleaned
from the formula for $m_A(U_\mu,
W_\omega)$ in Section \ref{1}, when
calculating relative multiplicities between
local pointed groups for concrete
examples, non-local points of
$p$-subgroups can sometimes usefully
be brought into consideration. We
mention, without proof, that when $A$
is the principal $2$-block algebra of the
symmetric group $S_5$, there exist local
pointed groups $U_\mu \leq W_\omega$
and $U < V < W$ such that all the points
$\nu$ of $V$ on $A$ satisfying
$U_\mu < V_\nu < W_\omega$
are non-local.

We shall be needing a lemma
describing how isofusions restrict
to pointed subgroups.

\begin{lem}
\label{4.4}
Let $U_\mu$ and $V_\nu$ be
pointed groups on $A$. Let $\phi :
V_\nu \rightarrow U_\mu$ be an
isofusion. Then the pointed subgroups
$S_\sigma$ of $U_\mu$ and the
pointed subgroups $T_\tau$ of
$V_\nu$ are in a bijective correspondence
such that $S_\sigma \leftrightarrow
T_\tau$ if and only if $\phi$ restricts
to an isofusion $T_\tau \rightarrow
S_\sigma$. Furthermore, when
$S_\sigma \leftrightarrow T_\tau$,
we have $m_A(S_\sigma, U_\mu) =
m_A(T_\tau, V_\nu)$ and $\phi$
restricts to an isomorphism
$N_V(T_\tau) \rightarrow N_U(S_\sigma)$.
\end{lem}

\begin{proof}
Fix a pointed subgroup $T_\tau \leq
V_\nu$. Define $S = \phi(T)$. We shall
show that there exists a point $\sigma$
of $S$ on $A$ such that $m_A(S_\sigma,
U_\mu) \geq m_A(T_\tau, V_\nu)$ and
$\phi$ restricts to an isofusion $T_\tau
\rightarrow S_\sigma$ and to a
monomorphism $N_V(T_\tau) \rightarrow
N_U(S_\sigma)$. That will suffice,
because the inequality will imply that
$S_\sigma \leq U_\mu$, whereupon
replacement of $\phi$ by $\phi^{-1}$
and an appeal to Lemma \ref{4.2} will
yield the non-strictness of the inequality,
the bijectivity of the injection $T_\tau
\mapsto S_\sigma$ and the bijectivity
of the monomorphism $N_V(T_\tau)
\rightarrow N_U(S_\sigma)$.

Let $i \in \mu$ and $j \in \nu$. Let
$(c, d)$ be a $\phi$-witness
$j \mapsto i$. Choose $k \in \tau$
such that $k \leq j$. It is easily checked
that $c k d$ is a primitive idempotent of
$A^S$ and $(ck, kd)$ is a $\phi$-witness
$k \mapsto ckd$. Moreover,
$ckd \leq i$. So, letting $\sigma$ be
the point of $S$ on $A$ owning $ckd$,
then $S_\sigma \leq U_\mu$. By the
rider of Lemma \ref{4.1}, $\sigma$ is
independent of the choice of $k$. So,
given any set $\{ k_1, ..., k_m \}$ of
mutually orthogonal elements of
$\tau \cap j A^T j$, then $\{ c k_1 d,
... , c k_m d \}$ is a set of mutually
orthogonal elements of $\sigma \cap
i A^S i$. Taking $m$ to be as large as
possible, we deduce the required
inequality of multiplicities.

It remains only to show that, given
$g \in N_V(T_\tau)$, then $\phi(g)
\in N_U(S_\sigma)$. Plainly, $\phi(g)
\in N_U(S)$. So ${}^{\phi(g)} (ckd)$ is
a primitive idempotent of $A^S$. But
${}^{\phi(g)} (ckd) = c . {}^g k. d$, so
$(c . {}^g k, {}^g k . d)$ is a
$\phi$-witness ${}^g k \mapsto
{}^{\phi(g)} (ckd)$. But
${}^g k \in \tau$. So, by the rider of
Lemma \ref{4.1} again, ${}^{\phi(g)}
(ckd) \in \sigma$. Therefore,
$\phi(g) \in N_U(S_\sigma)$, as
required.
\end{proof}

Notwithstanding our comments above
on the importance of non-local points
of $p$-subgroups, the rest of this
section is concerned with a
characterization of isofusions that
pertains only to local points.
We shall be making use of Brauer maps,
as defined in \cite[5.4.2, 5.4.10]{Lin18}
for $\cO G$-modules and, in particular,
for $G$-algebras. For a $p$-subgroup
$P \leq G$ and $\cO G$-module $M$,
we write the $P$-relative Brauer map
on $M$ as $\br_P : M^P \rightarrow M(P)$.
When an element of the Brauer quotient
$M(P)$ is written in the form $\oo{x}$,
it is to be understood that $\oo{x} =
\br_P(x)$ and $x \in M^P$. We shall
sometimes employ the overbar just as a
reminder that $\oo{x}$ is an element of
a Brauer quotient, and we shall not always
need to consider the choice of lift $x$.

Let $P$ and $Q$ be $p$-subgroups of
$G$ and let $\phi$ be a group isomorphism
$Q \rightarrow P$. Viewing $A$ as an
$\cO(G {\times} G)$-module, we write the
$\Delta(\phi)$-relative Brauer map on $A$
as $\br_\phi : A^\phi \rightarrow A(\phi)$.
Let $R$ be another $p$-subgroup of $G$
and let $\psi : R \rightarrow Q$ be an
isomorphism. The multiplication operation
$A \times A \rightarrow A$ restricts to an
$\cO$-bilinear map $A^\phi \times A^\psi
\rightarrow A^{\phi \psi}$ which
induces an $\FF$-bilinear map
$A(\phi) \times A(\psi) \rightarrow
A(\phi \psi)$, written $(\oo{u}, \oo{v})
\mapsto \oo{u} * \oo{v}$, where
we define $\oo{u} * \oo{v} = \oo{uv}$.
To see that $*$ is well-defined, observe
that, by an application of the Frobenius
relations, $uv \in \ker(\br_{\phi \psi})$
whenever $u \in \ker(\br_\phi)$ or
$v \in \ker(\br_\psi)$. We call $*$ a
{\bf localized multiplication}. Some
similar bilinear maps appear in
Puig--Zhou \cite[3.2]{PZ07}.

The localized multiplications, taken
together, inherit an associativity property,
to wit, given yet another $p$-subgroup
$S$ of $G$ and a group isomorphism
$\chi : S \rightarrow R$, then the expression
$\oo{u} * \oo{v} * \oo{w}$ is unambiguous
for any $\oo{u} \in A(\phi)$, $\oo{v} \in
A(\psi)$, $\oo{w} \in A(\chi)$, in fact,
$\oo{u} * \oo{v} * \oo{w} = \oo{uvw}$.
Also note that, putting $\phi = \id_P$, the
localized multiplication $A(P) {\times} A(P)
\rightarrow A(P)$ is the usual multiplication
on the Brauer quotient $A(P)$.

\begin{lem}
\label{4.5}
Let $P_\gamma$ and $Q_\delta$ be
local pointed groups on $A$. Let
$i \in \gamma$ and $j \in \delta$. Let
$\phi : Q \rightarrow P$ be a group
isomorphism. Then $\phi$ is an isofusion
$Q_\delta \rightarrow P_\gamma$ if
and only if there exist $\oo{s} \in
\oo{i} * A(\phi) * \oo{j}$ and $\oo{t} \in
\oo{j} * A(\phi^{-1}) * \oo{i}$ such that
$\oo{i} = \oo{s} * \oo{t}$ and
$\oo{j} = \oo{t} * \oo{s}$.
\end{lem}

\begin{proof}
One direction is clear. Conversely, suppose
such $\oo{s}$ and $\oo{t}$ exist. Let
$b = isj$ and $y = jti$, which belong to
$i A^\phi j$ and $j A^{\phi^{-1}} i$,
respectively. Then $\oo{by} = \oo{i}$,
hence
$$i - by \in i A^P i \cap \ker(\br_P)
  \subseteq J(i A^P i) \; .$$
But $i A^P i$ is a local ring, so $by$ has
an inverse $a$ in $i A^P i$. Putting
$x = ab$, then $x \in i A^\phi j$ and
$i = xy$. We have $(yx)^2 = yix = yx$,
which belongs to the local ring $j A^Q j$.
So $j = yx$. We have shown that $(x, y)$
is a $\phi$-witness $j \mapsto i$.
\end{proof}

\section{Uniformity}
\label{5}

For any interior $G$-algebra $A$, we
shall introduce a category $\cF^{\rm uni}(A)$
whose objects are the $p$-subgroups
of $G$ and whose morphisms are some
group monomorphisms. In subsequent
sections, strong further hypotheses will
be imposed. Our motive for the present
generality is not any anticipated breadth
of application. The removal of irrelevant
conditions is simply for the sake of clarity.

\begin{thm}
\label{5.1}
Let $A$ be an interior $G$-algebra. Let
$P$ and $Q$ be $p$-subgroups of $G$
and let $\phi : Q \rightarrow P$ be a
group isomorphism. Then the following
three conditions are equivalent:

\noin {\bf (a)} The localized multiplications
$* : A(\phi) \times A(\phi^{-1})
\rightarrow A(P)$ and $* : A(\phi^{-1}) \times
A(\phi) \rightarrow A(Q)$ are surjective.

\noin {\bf (b)} There exist $\oo{u} \in
A(\phi)$ and $\oo{v} \in A(\phi^{-1})$
such that $1_{A(P)} = \oo{u} * \oo{v}$
and $1_{A(Q)} = \oo{v} * \oo{u}$.

\noin {\bf (c)} There is a bijective
correspondence between the local
points $\gamma$ of $P$ on $A$ and
the local points $\delta$ of $Q$ on
$A$ whereby $\gamma \leftrightarrow
\delta$ provided $\phi$ is an isofusion
$Q_\delta \rightarrow P_\gamma$.
Furthermore, $m_A(P_\gamma) =
m_A(Q_\delta)$ when $\gamma
\leftrightarrow \delta$.
\end{thm}

\begin{proof}
We first deal with a degenerate case.
Suppose $A(P) = 0$ or $A(Q) = 0$.
Given $u \in A^\phi$ then $\br_\phi(u) =
\br_P(1) * \br_\phi(u) * \br_Q(1)$. So
$A(\phi) = 0$ and $A(\phi^{-1}) = 0$. It
is now easy to deduce that, if at least
one of (a), (b), (c) holds, then both
$A(P)$ and $A(Q)$ vanish, hence all
three of (a), (b), (c) hold. So we may
suppose that $A(P)$ and $A(Q)$ are
non-zero.

Assuming (a), we shall deduce (b). By
the surjectivity of one of the local
multiplication operations,
$1_{A(P)} = \oo{u} * \oo{v}$ for some
$\oo{u} \in A(\phi)$ and $\oo{v} \in
A(\phi^{-1})$. We shall show that
$1_{A(Q)} = \oo{v} * \oo{u}$ for any
such $\oo{u}$ and $\oo{v}$. Given
$\oo{a} \in A(P)$, then $\oo{a} =
\oo{a} * \oo{u} * \oo{v}$ and
$\oo{a} * \oo{u} \in A(\phi)$. So the
$\FF$-linear map $\dash * \oo{v} :
A(\phi) \rightarrow A(P)$ is surjective.
Similarly, $\oo{u} * \dash : A(Q)
\rightarrow A(\phi)$ is surjective.
Therefore,
$$\dim_\FF(A(P)) \leq
  \dim_\FF(A(\phi)) \leq
  \dim_\FF(A(Q)) \; .$$
Replacing $\phi$ with $\phi^{-1}$
yields two further inequalities, hence
$$\dim_\FF(A(P)) = \dim_\FF(A(\phi))
  = \dim_\FF(A(Q))
  = \dim_\FF(A(\phi^{-1})) \; .$$
Since $1_{A(P)} = \oo{u} * \oo{v} * \oo{u}
* \oo{v}$, the composite $\FF$-map
$$(\oo{u} * \dash) \comp
  (\oo{v} * \dash) \comp (\oo{u} * \dash)
  \comp (\oo{v} * \dash) \: : \: A(\phi)
  \leftarrow A(Q) \leftarrow A(\phi)
  \leftarrow A(Q) \leftarrow A(\phi)$$
is the identity on $A(\phi)$. By
considering dimensions, the composite
$$\oo{v} * \oo{u} * \dash =
  (\oo{v} * \dash) \comp (\oo{u} * \dash)
  \: : \: A(Q) \leftarrow A(\phi)
  \leftarrow A(Q)$$
must be an $\FF$-isomorphism. But
$\oo{v} * \oo{u}$ is easily seen to be
an idempotent of $A(Q)$. Therefore,
$1_{A(Q)} = \oo{v} * \oo{u}$. We
have deduced (b).

Now assuming (b), we shall deduce (c).
Consider a primitive idempotent
decomposition $1 =
\sum_{j \in J} j$ in $A^Q$. For each
$j \in J$, write $\oo{j} = \br_Q(j)$.
Let $J_0 = \{ j \in J : \oo{j} \neq 0 \}$.
For each $j \in J_0$, the element
$\oo{u} * \oo{j} * \oo{v} \in A(P)$ is a
primitive idempotent, which we lift to
a primitive idempotent $i_j \in A^P$.
We have $1 = \sum_{j \in J_0} \oo{i_j}$
and $1 = \sum_{j \in J_0} \oo{j}$ as
primitive idempotent decompositions
in $A(P)$ and $A(Q)$, respectively.

We claim that, for every local point
$\delta$ of $Q$ on $A$, there exists a
local point $\gamma$ of $P$ on $A$
such that $\phi : Q_\delta \rightarrow
P_\gamma$. Plainly, $\delta$ intersects
with $J_0$. Choose $j \in \delta \cap J_0$,
write $i = i_j$ and let $\gamma$ be the
local point of $P$ owning $i$. Defining
$s = i u j$ and $t = j v i$, then $\oo{s}$
and $\oo{t}$ satisfy the hypothesis of
Lemma \ref{4.5}. The claim is now
established.

Replacing $\phi$ with $\phi^{-1}$ and
applying Lemma \ref{4.2}, we deduce
that the condition $\phi : Q_\delta
\rightarrow P_\gamma$ characterizes a
bijective correspondence between the
local points $\gamma$ of $P$ on $A$
and the local points $\delta$ of $Q$ on
$A$. Suppose $\gamma \leftrightarrow
\delta$. We complete the deduction of
(c) by observing that
$$m_A(P_\gamma) = | \{ j \in J_0 :
  \oo{i_j} \in \br_P(\gamma) \} | =
  | \{ j \in J_0 : \oo{j} \in \br_Q(\delta)
  \} | = m_A(Q_\delta) \; .$$

Assuming (c), we shall deduce (b). Let
$1 = \sum_{i \in I} i$ and
$j = \sum_{j \in J} j$ be primitive
idempotent decompositions in $A^P$ and
$A^Q$, respectively. Define $J_0$ as
before and $I_0$ similarly. By the
assumption, there is a bijection $J_0 \ni
j \mapsto i_j \in I_0$ such that, for each
$j \in J_0$, there exists a $\phi$-witness
$(u_j, v_j) : j \mapsto i_j$. Let
$u = \sum_{j \in J_0} u_j$ and
$v = \sum_{j \in J_0} v_j$. Then
$uv = \sum_{j \in J_0} i_j$ and
$vu = \sum_{j \in J_0} j$. So
$$\oo{u} * \oo{v} = \sum_{j \in J_0} \oo{i_j}
  = \sum_{i \in I_0} \oo{i} = 1_{A(P)}$$
and $\oo{v} * \oo{u} = 1_{A(Q)}$. We have
deduced (b).

Finally, assuming (b), then any
$\oo{t} \in A(P)$ satisfies $\oo{t} =
\oo{t} * 1_{A(P)} = \oo{tu} * \oo{v}$, so
the localized multiplication $A(\phi)
\times A(\phi^{-1}) \rightarrow A(P)$
is surjective, likewise the localized
multiplication $A(\phi^{-1}) \times
A(\phi) \rightarrow A(Q)$. We have
deduced (a).
\end{proof}

When the three equivalent conditions
in Theorem \ref{5.1} hold, we call
$\phi$ {\bf locally $A$-uniform}.

\begin{thm}
\label{5.2}
Let $A$ and $\phi : Q \rightarrow P$ be as
in the previous theorem. Then the following
three conditions are equivalent:

\noin {\bf (a)} Every isomorphism restricted
from $\phi$ is locally $A$-uniform.

\noin {\bf (b)} There is a bijective
correspondence between the
points $\alpha$ of $P$ on $A$ and
the points $\beta$ of $Q$ on
$A$ whereby $\alpha \leftrightarrow
\beta$ provided $\phi$ is an isofusion
$Q_\beta \rightarrow P_\alpha$.
Furthermore, $m_A(P_\alpha) =
m_A(Q_\beta)$ when $\alpha
\leftrightarrow \beta$.

\noin {\bf (c)} We have $A^\times
\cap A^\phi \neq \emptyset$.
\end{thm}

\begin{proof}
We shall show that each condition
implies the next and the last implies
the first.

Assume (a). Let $\beta$ be a point of
$Q$ on $A$. Consider a local pointed
subgroup $T_\tau \leq Q_\beta$. Let
$S = \phi(T)$. By the assumption, there
exists a unique local point $\sigma$ of
$S$ on $A$ such that $\phi$ restricts to
an isofusion $T_\tau \rightarrow S_\sigma$.
Let $\cV$ be the set of points $\nu$ of
$Q$ on $A$ such that $T_\tau \leq Q_\nu$.
Let $\cU$ be the set of points $\mu$ of
$P$ on $A$ such that $S_\sigma \leq
P_\mu$. We claim that there is a bijection
$\cU \leftrightarrow \cV$ such that
$\mu \leftrightarrow \nu$ if and only if
$\phi$ is an isofusion $Q_\nu \rightarrow
P_\mu$, moreover, when those
equivalent conditions hold,
$m_A(P_\mu) = m_A(Q_\nu)$. Since
$\beta \in \cV$, the claim implies the
existence of a point $\alpha$ of $P$
on $A$ such that $\phi$ is an isofusion
$Q_\beta \rightarrow P_\alpha$ and
$m_A(P_\alpha) = m_A(Q_\beta)$.
Again replacing $\phi$ with $\phi^{-1}$
and applying Lemma \ref{4.2}, part (d)
will follow. (Note that, in the application
of the claim, $T_\tau$ does not appear
in the conclusion. The role of $T_\tau$
is to supply an opportunity for an
inductive proof of the claim.)

To demonstrate the claim, we argue by
induction on $|Q : T|$. By the assumption,
$m_A(S_\sigma) = m_A(T_\tau)$. That is,
$$\sum_{\mu \in \cU} m_A(S_\sigma,
  P_\mu) \, m_A(P_\mu)
  = \sum_{\nu \in \cV} m_A(T_\tau,
  Q_\nu) \, m_A(Q_\nu) \; .$$
Let $\cV'$ be the set of $\nu \in \cV$
such that $T_\tau$ is not a defect
pointed subgroup of $V_\nu$. Define
$\cU' \subseteq \cU$ similarly for
$S_\sigma$. Given $\nu\ \in \cV'$,
then $Q_{\nu'}$ has a defect pointed
subgroup $T'_{\tau'}$ strictly containing
$T_\tau$. Let $S' = \phi(T')$. By the
assumption, there exists a unique
local point $\sigma'$ of $S'$ on $A$
such that $\phi$ restricts to an isofusion
$T'_{\tau'} \rightarrow S'_{\sigma'}$. By
the inductive hypothesis, there exists
a point $\mu'$ of $P$ on $A$ such
that $\phi$ is an isofusion $Q_{\nu'}
\rightarrow P_{\mu'}$ and
$m_A(P_{\mu'}) = m_A(Q_{\nu'})$.
By Lemma \ref{4.4}, $S_\sigma <
S'_{\sigma'} \leq P_{\mu'}$. So
$\mu' \in \cU'$. The same lemma
also implies that $m_A(S_\sigma,
P_{\mu'}) = m_A(T_\tau, Q_{\nu'})$.
Replacing $\phi$ with $\phi^{-1}$,
we deduce that there is a bijective
correspondence $\cU' \ni \mu'
\leftrightarrow \nu' \in \cV'$
characterized by the condition that
$\phi$ is an isofusion $Q_{\nu'}
\rightarrow P_{\mu'}$. Therefore,
$$\sum_{\mu' \in \cU'} m_A(S_\sigma,
  P_{\mu'}) \, m_A(P_{\mu'})
  = \sum_{\nu' \in \cV'} m_A(T_\tau,
  Q_{\nu'}) \, m_A(Q_{\nu'}) \; .$$
Comparing with the similar equality
established earlier, we deduce that
$\cU = \cU'$ if and only if $\cV = \cV'$.
In that case, the claim is now clear. So
we may assume that $\cU' \subset \cU$
and $\cV' \subset \cV$. By Theorem
\ref{3.1}, the differences are singleton,
say, $\cU - \cU' = \{ \mu \}$ and
$\cV - \cV' = \{ \nu \}$. So
$$m_A(S_\sigma, P_\mu) \,
  m_A(P_\mu) = m_A(T_\tau, Q_\nu)
  m_A(Q_\nu) \; .$$
By Lemma \ref{4.4}, $|N_P(S_\sigma)|
= |N_Q(T_\tau)|$. Since $S_\sigma$
and $T_\tau$ are defect pointed
subgroups of $P_\mu$ and $Q_\nu$,
respectively, Remark \ref{3.2} yields
$$m_A(S_\sigma, P_\mu) =
  |N_P(S_\sigma) : S| = |N_Q(T_\tau) : T|
  = m_A(T_\tau, Q_\nu) \; .$$
Therefore, $m_A(P_\mu) = m_A(Q_\nu)$.
The claim is demonstrated. The
deduction of (b) is complete.

Assume (b). Again, let
$1 = \sum_{j \in J} j$ be a primitive
idempotent decomposition in $A^Q$.
By the assumption, there is a
primitive idempotent decomposition
$1 = \sum_{j \in J} i_j$ in $A^P$ such
that, writing $\alpha$ and $\beta$ for
the points $P$ and $Q$ owning $i_j$
and $j$, respectively, then $\phi$ is an
isofusion $Q_\beta \rightarrow P_\alpha$.
For each $j \in J$, let $(s_j, t_j)$ be a
$\phi$-witness $j \mapsto i_j$. Let
$s = \sum_{j \in J} s_j$ and $t =
\sum_{j \in J} t_j$. A short calculation
yields $st = 1 = ts$. In particular,
$s \in A^\times \cap A^\phi$. We
have deduced (c).

Finally, assume (c). Let $u \in A^\times
\cap A^\phi$. Define $v = u^{-1}$. Then
$v \in A^\times \cap A^{\phi^{-1}}$ and
the elements $\uu{u} \in A(\phi)$ and
$\uu{v} \in A(\phi^{-1})$ satisfy
$1_{A(P)} = \uu{u} * \uu{v}$ and
$1_{A(Q)} = \uu{v} * \uu{u}$. We have
deduced (a).
\end{proof}

When the three equivalent conditions in
Theorem \ref{5.2} hold, we say that
$\phi$ is {\bf $A$-uniform}. We form a
category $\cF^{\rm uni}(A)$ such that the
objects are the $p$-subgroups
of $G$ and the morphisms are the
composites of the inclusions and the
$A$-uniform isomorphisms. For any
Sylow $p$-subgroup $S$ of $G$, we
let $\cF_S^{\rm uni}(A)$ denote the
full subcategory of $\cF^{\rm uni}(A)$
such that the objects of
$\cF_S^{\rm uni}(A)$ are the subgroups
of $S$. Plainly, $\cF_S^{\rm uni}(A)$ is
a fusion system on $S$, and
$\cF_S^{\rm uni}(A)$ is well-defined up
to isomorphism, in fact, well-defined up
to $G$-conjugation, independently of
the choice of $S$. We call
$\cF_S^{\rm uni}(A)$ the {\bf uniform
fusion system} of $A$ on $S$.

\section{Bifree bipermutation algebras}
\label{6}

Let $D$ be a finite $p$-group. We shall
discuss the uniform fusion system
$\cF^{\rm uni}(A)$ for various kinds of
interior $D$-algebra $A$. We shall be
imposing hypotheses as we need them.
Again, the generality is in the service of
clarity. We shall find, in Section \ref{8}, that
whenever $\cF^{\rm uni}(A)$ comes to be
of concern in our contexts of  application,
it will coincide with the usual fusion system.
For the time-being, though, the logic of our
discussion requires us to treat
$\cF^{\rm uni}(A)$ as a potential innovation.
That having been said, Question \ref{6.2}
does raise a possibility that the generality
may be of interest in itself.

We write $\cI(D)$ to denote the set of
group isomorphisms $\phi$ such that
$\cod(\phi)$ and $\dom(\phi)$ are subgroups
of $D$. Observe that the condition
$\Delta(\phi) = N$ characterizes a bijective
correspondence $\phi \leftrightarrow N$
between the isomorphisms $\phi$ in
$\cI(D)$ and those subgroups
$N \leq D {\times} D$ that intersect trivially
with $D {\times} 1$ and $1 {\times} D$.
Let $\Omega$ be a $(D, D)$-biset. We call
$\Omega$ {\bf bifree} provided $D {\times} 1$
and $1 {\times} D$ act freely on $\Omega$.
Thus, $\Omega$ is bifree if and only if, for all
$w \in \Omega$, there exists $\phi \in \cI(D)$
such that $N_{D \times D}(w) = \Delta(\phi)$.

We define a {\bf bipermutation
$D$-algebra} to be an interior
$D$-algebra $A$ such that $A$ has
a $D {\times} D$-stable basis.
Suppose $A$ is a bipermutation
$D$-algebra and let $\Omega$
be a $D {\times} D$-stable basis.
Generalizing a comment made in
Section \ref{1}, we note that, by
Green's Indecomposability Theorem
and the Krull--Schmidt Theorem,
$\Omega$ is well-defined as a
$(D, D)$-biset up to isomorphism.
When the left action of $D$ and the
right action of $D$ on $\Omega$ are
free, we say that $A$ is {\bf bifree}.

\begin{thm}
\label{6.1}
Let $A$ be a bifree bipermutation
$D$-algebra. Then the following two
conditions are equivalent:

\noin {\bf (a)} Given
$\phi \in \cI(D)$ such that
$A(\phi) \neq 0$, then $\phi$ is
$A$-uniform.

\noin {\bf (b)} $A$ has a unital
$D {\times} D$-stable basis.
\end{thm}

\begin{proof}
Let $\Omega$ be a $D {\times} D$-stable
basis for $A$. Given
$\phi \in \cI(D)$, then $A(\phi) \neq 0$
if and only if $\Omega^{\Delta(\phi)} \neq
\emptyset$. So (b) implies (a). Conversely,
suppose (a) holds. By Theorem \ref{2.4},
(b) will follow when we have shown that
given $w \in \Omega$, then
$N_{D {\times} D}(w)$ fixes a unit of $A$.
By the bifreeness, $N_{D {\times} D}(w) =
\Delta(\phi)$ for some $\phi \in \cI(D)$.
We have $A(\phi) \neq 0$ so,
by the assumption, $\phi$ is $A$-uniform
and $\Delta(\phi)$ fixes a unit of $A$. We
have deduced (b).
\end{proof}

When the two equivalent conditions in
the latest theorem hold, we say that the
bifree bipermutation $D$-algebra $A$
is {\bf uniform}. The term {\it uniform
bifree bipermutation $G$-algebra} is
rather long, but a decision on how to
improve it might best await resolution
of the following question.

\begin{ques}
\label{6.2}
Is every bifree bipermutation $D$-algebra
uniform?
\end{ques}

We see no reason to expect the affirmative
answer, but we have been unable to find
a counter-example. Consider a finite
$D$-set $\Gamma$. It is an easy exercise
to show that, if $D$ is non-cyclic then, for
some $\Gamma$, the bipermutation
$D$-algebra $\End_\cO(\cO \Gamma)$
does not have a unital $D {\times} D$-stable
basis. However, it is also easy to show that,
given $D$ and $\Gamma$ such that
$\End_\cO(\cO \Gamma)$ does not have
a unital $D {\times} D$-stable basis, then
$\End_\cO(\cO \Gamma)$ is not bifree.

\begin{pro}
\label{6.3}
Let $e$ be an idempotent of $Z(\cO G)$.
Let $S$ be a Sylow $p$-subgroup of $G$.
Let $A = \cO G e$ as an interior
$S$-algebra by restriction. Then $A$ is
a uniform bifree bipermutation
$S$-algebra and $\cF^{\rm uni}(A)$
coincides with the fusion system 
$\cF_S(G)$ of $G$ on $S$.
\end{pro}

\begin{proof}
Corollary \ref{2.5} says that $A$ is a
uniform bifree bipermutation $S$-algebra.
The proof of the corollary shows the
required equality of fusion systems.
\end{proof}

In particular, every group fusion system
can be realized as the uniform fusion
system of a unform bifree bipermutation
algebra. At the time of writing, the authors
know nothing about whether some or all
exotic fusion systems can be realized in
that way.

As in Gelvin--Reeh \cite[3.4]{GR15}, a
bifree $(D, D)$-biset $\Omega$ is
said to be {\bf $\cF$-semicharacteristic}
provided the following condition holds:
given an isomorphism $\phi$ in $\cI(D)$,
then $\phi$ is an $\cF$-isomorphism if
and only if $\Omega^{\Delta(\phi)} \neq
\emptyset$, furthermore, in that case,
$$|\Omega^{\Delta({\rm cod}(\phi))}|
  = |\Omega^{\Delta(\phi)}| =
  |\Omega^{\Delta({\rm dom}(\phi))}| \; .$$
It is easy to see that,
given an $\cF$-semicharacteristic
$(D, D)$-biset $\Omega$, then the
integer $|\Omega|/|D|$ is coprime to
$p$ if and only if $\Omega$ is
$\cF$-characteristic as defined in
Linckelman \cite[8.7.9]{Lin18}.

\begin{thm}
\label{6.4}
Let $A$ be a uniform bifree bipermutation
$D$-algebra. Then every
$D {\times} D$-stable basis for $A$ is an
$\cF^{\rm uni}(A)$-semicharacteristic
$(D, D)$-biset.
\end{thm}

\begin{proof}
Let $\Omega$ be any $D {\times} D$-stable
basis for $A$ and let $\phi : Q \rightarrow P$
be an $\cF^{\rm uni}(A)$-isomorphism. By
the hypothesis on $A$, there exists
$u \in A^\times \cap A^\phi$. The
bijection $u \Omega \ni u w \leftrightarrow
w \in \Omega$ restricts to a bijection
$(u \Omega)^{\Delta(\phi)} \leftrightarrow
\Omega^Q$. So $|(u \Omega)^{\Delta(\phi)}|
= |\Omega^Q|$. But $u \Omega \cong
\Omega$ as $(P, Q)$-bisets. So
$|\Omega^{\Delta(\phi)}| = |\Omega^Q|$,
and similarly for $|\Omega^P|$.
\end{proof}

Theorem \ref{6.4}, together with
Proposition \ref{6.3}, recovers the main
result of Gelvin \cite{Gel19}, which asserts
that, in the notation of the proposition, any
$S {\times} S$-stable basis of $\cO G b$ is
$\cF_S(G)$-semicharacteristic.

In Proposition \ref{8.1}, we shall confirm that,
for uniform almost-source algebras, the
uniform fusion system discussed in the present
section coincides with the fusion system
discussed in Section \ref{1}. Incidentally, that
will realize Theorem \ref{1.1} as a special case
of Theorem \ref{6.4}.

\section{Divisibility of some stable bases}
\label{7}

We shall be discussing the fusion system
of an almost-source algebra and
Linckelmann's characterization of the fusion
system in terms of a stable basis of the
almost-source algebra. The characterization
mentions a condition that is not an
isomorphism invariant, namely, the
condition that a given object is fully
centralized. We shall show that, in some
cases, that condition can be omitted.

For the moment, let $D$ be any finite
$p$-group, let $\cF$ be a fusion system
on $D$ and let $\Omega$ be a
$(D, D)$-biset. Recall, from Section \ref{1},
that $\Omega$ is {\bf $\cF$-divisible}
provided $\Omega$ is bifree and, given
any isomorphism $\phi$ in $\cI(D)$, then
$\phi$ is an $\cF$-isomorphism if and only
if $\Delta(\phi)$ fixes an element of
$\Omega$.

Now let $D$ be the defect group of a
block $b$ of $\cO G$. For the convenience
of the reader, let us quickly extract a few
definitions from Linckelmann
\cite[Sections 6.4, 8.7]{Lin18}. For each
$P \leq D$, we make the usual identification
$(\cO G)(P) = \FF C_G(P)$. Let $\ell$ be an
idempotent of $(\cO G b)^D$ such that
$\br_D(\ell) \neq 0$ and, for all $P \leq D$,
we have an inclusion of idempotents
$\br_P(\ell) \leq b_P$ for some block
$b_P$ of $\FF C_G(P)$. Recall, we have
an inclusion of Brauer pairs
$(P, b_P) \leq (D, b_D)$. The interior
$D$-algebra $A = \ell \cO G \ell$ is
called an {\bf almost-source algebra}
of $\cO G b$. When $\ell$ is a primitive
idempotent of $(\cO G b)^D$, we call
$A$ a {\bf source algebra} of $\cO G b$.
For arbitrary $A$, the fusion system
$\cF$ associated with $A$ is defined to
be the fusion system on $D$ associated
with the Brauer pair $(D, b_D)$. Thus, for
$P, Q \leq D$, the $\cF$-isomorphisms
$Q \rightarrow P$ are the conjugation
isomorphisms $Q \ni y \mapsto {}^g y$,
where $g \in G$ satisfying
$(P, b_P) = {}^g (Q, b_Q)$.

The next remark is well-known.

\begin{rem}
\label{7.1}
Let $b$ be a block of $\cO G$ with defect
group $D$ and almost-source algebra $A$.
Let $\Omega$ be a $D {\times} D$-stable
basis for $A$. Then $\Omega$ is isomorphic
to its opposite biset and
$|\Omega^{\Delta(\phi)}| =
|\Omega^{\Delta(\phi^{-1})}|$ for any
$\phi \in \cI(D)$.
\end{rem}

\begin{proof}
This follows from Gelvin
\cite[Proposition 6, Lemma 7]{Gel19}.
\end{proof}

The next result is Linckelmann
\cite[8.7.1]{Lin18}.

\begin{thm}
\label{7.2}
{\rm (Linckelmann.)}
Let $b$ be a block of $\cO G$ with defect
group $D$ and almost-source algebra $A$.
Let $\cF$ be the fusion system on $D$
associated with $A$. Let $\Omega$ be a
$D {\times} D$-stable basis for $A$.
Given an isomorphism $\phi$ in $\cI(D)$,
if $\Delta(\phi)$ fixes a point of $\Omega$,
then $\phi$ is an $\cF$-isomorphism,
furthermore, the converse holds when
$\dom(\phi)$ or $\cod(\phi)$ is fully
$\cF$-centralized.
\end{thm}

Can we omit the clause on the domain
or codomain being fully $\cF$-centralized?
More precisely, given $b$, must $\Omega$
be divisible for some almost-source algebra
$A$ of $\cO G b$? Conjecture \ref{1.3}
implies the affirmative. Given $b$, must
$\Omega$ be divisible for some source
algebra and hence for every almost-source
algebra of $\cO G b$? To that, Conjecture
\ref{1.6} implies the affirmative.

In the next section, we shall be discussing
some cases where $\Omega$ has the
stronger property of being
$\cF$-semicharacteristic. In the rest of
the present section, we shall be proving
the $\cF$-divisibility of $\Omega$ in
some cases where we have been unable
to prove the stronger condition.

Following Craven \cite[3.69]{Cra11}, we
call $\cF$ {\bf constrained} provided
$C_D(O_p(\cF)) = Z(O_p(\cF))$. By
\cite[4.46]{Cra11}, if $\cF$ is constrained,
then $O_p(\cF)$ is the minimum among
the $\cF$-radical $\cF$-centric subgroups
of $D$.

\begin{pro}
\label{7.3}
In the notation of Theorem \ref{7.2},
if $\cF$ is $p$-constrained, then
$\Omega$ is $\cF$-divisible.
\end{pro}

\begin{proof}
We must show that, given $P, Q \leq D$
and an $\cF$-isomorphism $\phi : Q
\rightarrow P$, then $\Delta(\phi)$ fixes
an element of $\Omega$. Let
$R = O_p(\cF)$. Then $\phi$ extends
to an $\cF$-isomorphism $\psi : QR
\rightarrow PR$. Since $R$ is $\cF$-centric,
$PR$ and $QR$ are $\cF$-centric and
hence fully $\cF$-centralized. So
$\Omega^{\Delta(\psi)} \neq \emptyset$.
In particular, $\Omega^{\Delta(\phi)}
\neq \emptyset$.
\end{proof}

By \cite[5.95]{Cra11}, the group fusion
system of a $p$-solvable group is
solvable. Part of \cite[5.90]{Cra11} says
that all saturated subsystems of a
solvable fusion system are solvable.
So, if $G$ is $p$-solvable, then
$\cF$ is solvable. A theorem
of Aschbacher in \cite[5.91]{Cra11}
asserts that every solvable fusion system
is constrained. Therefore, Proposition
\ref{7.3} has the following corollary.

\begin{cor}
\label{7.4}
In the notation of Theorem \ref{7.2}, if
$G$ is $p$-solvable, then $\Omega$
is $\cF$-divisible.
\end{cor}

Alternatively, the corollary can be deduced
from Proposition \ref{7.3} by using
\cite[8.1.8]{Lin18} to reduce to the
case where $O_{p'}(G) = 1$, then applying
a theorem of Hall--Higman in
\cite[5.93]{Cra11}.

To present one more case where we
can guarantee divisibility, we need some
preliminaries. Consider a $p$-subgroup
$P$ of $G$. Recall, the condition
$\br_P(\gamma)V \neq 0$ characterizes
a bijective correspondence $\gamma
\leftrightarrow [V]$ between the local
points $\gamma$ of $P$ on $\cO G$
and the isomorphism classes $[V]$ of
simple $\FF C_G(P)$-modules $V$. Let
$\sigma_P^G$ denote the local point
of $P$ on $\cO G$ corresponding to
the trivial $\FF C_G(P)$-modules. The
next result, well-known, can be viewed
as a version of Brauer's Third Main
Theorem. A proof of it is implicit in
Linckelmann \cite[6.3.14]{Lin18}. For
convenience, we extract the argument.

\begin{thm}
\label{7.5}
Let $Q$ and $P$ be $p$-subgroups of
$G$. Then $Q \leq P$ if and only if
$Q_{\sigma_Q^G} \leq P_{\sigma_P^G}$.
\end{thm}

\begin{proof}
We may assume that $\cO = \FF$ and
$Q \leq P$. Let $\eta_G : \FF G \rightarrow
\FF$ be the augmentation map. Given
an idempotent $e$ of $\FF G$, then
$\eta_G(e) = 1$ if and only if $e$ acts as
the identity on a trivial $\FF G$-module.
That is equivalent to the condition that
$k \leq e$ for some $k \in \sigma_1^G$.
So, letting $i \in \sigma_P^G$ and
$j \in \sigma_Q^G$, we have
$\eta_{C_G(P)}(i) = 1 = \eta_{C_G(Q)}(j)$.
For all $x \in (\FF G)^P$, we have
$\eta_{C_G(P)}(\br_P(x)) = \eta_G(x)$ and
similarly with $Q$ in place of $P$. Hence
$$\eta_{C_G(Q)}(\br_Q(i)) = \eta_G(i)
  = \eta_{C_G(P)}(\br_P(i)) \; .$$
Therefore, $\eta_{C_G(Q)}(\br_Q(ij))
= \eta_{C_G(Q)}(\br_Q(i)) .
\eta_{C_G(Q)}(\br_Q(j)) = 1$. So
$ij \not\in J((\FF G)^Q)$. It follows that
$i \geq j'$ for some conjugate $j'$ of
$j$ in $(\FF G)^Q$.
\end{proof}

Note that, from Theorem \ref{7.5}, we
immediately recover the version of
Brauer's Third Main Theorem that is
explicit stated in \cite[6.3.14]{Lin18}:
if $b$ is the principal block of $\cO G$,
then $b_P$ is the principal block of
$\FF C_G(P)$ for each $P \leq D$. In
particular, if $b$ is the principal block,
then $D$ is a Sylow $p$-subgroup of
$G$ and $\cF = \cF_D(G)$, the fusion
system of $G$ on $D$.

\begin{thm}
\label{7.6}
In the notation of Theorem \ref{7.2},
if $b$ is the principal block of $\cO G$,
then $\Omega$ is $\cF$-divisible.
\end{thm}

\begin{proof}
Given $P, Q \leq D$ and $g \in G$ such
that $P = {}^g Q$, then $P_{\sigma_P^G}
= {}^g (Q_{\sigma_Q^G})$. So any
$\cF$-isomorphism $\phi :
Q \rightarrow P$ is an isofusion
$Q_{\sigma_Q^G} \rightarrow
P_{\sigma_P^G}$. By Theorem \ref{7.5},
$P_{\sigma_P^G} \leq D_{\sigma_D^G}$.
But $A \cap \sigma_D^G \neq
\emptyset$. So we can define
$\sigma_P^A = A \cap \sigma_P^G$
as a local point of $P$ on $A$, likewise
$\sigma_Q^A$. Since $\phi$ is an
isofusion $Q_{\sigma_Q^A} \rightarrow
P_{\sigma_P^A}$, Lemma \ref{4.5}
implies that $A(\phi)$ and
$A(\phi^{-1})$ are non-zero.
\end{proof}

The proof of Theorem \ref{1.8} is
complete.

\section{Uniformity of some
almost-source algebras}
\label{8}

Again, we let $D$ be a defect group of
a block $b$ of $\cO G$. We shall prove
Theorems \ref{1.2}, \ref{1.4}, \ref{1.5} and
Proposition \ref{1.7}.

First, though, we are now ready to note
that Conjecture \ref{1.3} implies
another characterization of the fusion
system of an almost-source algebra.

\begin{pro}
\label{8.1}
Let $\cF$ be the fusion system of a
uniform almost-source $D$-algebra
$A$ of $\cO G b$. Then $\cF =
\cF^{\rm uni}(A)$.
\end{pro}

\begin{proof}
We need only show that $\cF$ and
$\cF^{\rm uni}(A)$ have the same
isomorphisms. By Theorem \ref{7.2},
every $\cF^{\rm uni}(A)$-isomorphism
is an $\cF$-isomorphism. For the
converse, consider an $\cF$-isomorphism
$\phi$. Since $\cF$ is a saturated fusion
system, Alperin's Fusion Theorem
guarantees that $\phi = \phi_1 ...
\phi_n$ where $\phi_1$, $...$, $\phi_n$,
respectively, are isomorphisms restricted
from $\cF$-automorphisms $\psi_1$, ...,
$\psi_n$ of $\cF$-centric subgroups
of $D$. Letting $r_1 \in A^\times \cap
A^{\psi_1}$, $...$, $r_n \in A^\times
\cap A^{\psi_n}$, then $r_1 ... r_n \in
A^\times \cap A^\phi$. So $\phi$ is
an $\cF^{\rm uni}(A)$-isomorphism.
\end{proof}

Via the proposition, Theorem \ref{1.1}
is now recovered as a special case of
Theorem \ref{6.4}.

Let us prove Theorem \ref{1.2}. Since
the $\cF$-isomorphisms are closed
under restriction to group isomorphisms,
(a) and (b) are equivalent and they can
be expressed as:

\smallskip
\noin {\bf (y)} Every $\cF$-isomorphism
is uniform.

\smallskip
On the other hand, by Theorem
\ref{6.1}, condition (c) can be
reformulated as:

\smallskip
\noin {\bf (z)} Given $\phi \in \cI(D)$
such that $A(\phi) \neq 0$, then
$\phi$ is uniform.

\smallskip
\noin Theorem \ref{7.2} immediately shows
that (y) implies (z). Conversely, assume (z),
and let $\phi$ be an $\cF$-isomorphism.
Invoking Alperin's Fusion Theorem, write
$\phi = \phi_1 ... \phi_m$ where each
$\phi_t$ is an $\cF$-isomorphism restricted
from an $\cF$-automorphism $\psi_t$ of an
$\cF$-centric subgroup $P_t \leq D$. Each
$P_t$ is fully $\cF$-centralized. So, by
Theorem \ref{7.2} again, $A(\psi_t) \neq 0$.
By the assumption, $\psi_t$ is uniform.
Hence, each $\phi_t$ is uniform and
$\phi$ is uniform. We have deduced (y).
Theorem \ref{1.2} is now proved.

To deal with Theorem \ref{1.5}, just a few
words of explanation are needed. We now
suppose that $A$ is a souce $D$-algebra.
The point $\lambda_D$ in the statement
of Theorem \ref{1.5} is the point of $D$ on
$\cO G$ owning the unity element $\ell$
of $A$. For $P \leq D$, the condition
$\beta' = A \cap \beta$ characterizes a
bijective correspondence between the
points $\beta'$ of $P$ on $A$ and the
points $\beta$ of $P$ on $\cO G$
satisfying $A \cap \beta \neq \emptyset$.
We have $m_A(P_{\beta'}) =
m_{\cO G}(P_\beta, D_{\lambda_D})$.
Of course, $\beta'$ is local if and only
if $\beta$ is local. It is now clear that
Theorem \ref{1.5} follows as a special
case of Theorem \ref{1.2}.

Proposition \ref{1.7} can be dispatched
quickly too. Still supposing that $A$ is
a source algebra of $\cO G b$, now
supposing also that $\cF = N_\cF(D)$,
let $\phi$ be an $\cF$-isomorphism. Let
$\psi$ be an extension of $\phi$ to an
$\cF$-automorphism of $D$. Let
$r \in A^\times$ such that $(r, r^{-1})$
is a $\psi$-witness $\ell \mapsto \ell$.
Then $r \in A^{\Delta(\psi)} \subseteq
A^{\Delta(\phi)}$. So, by Theorem
\ref{2.4}, $A$ has a unital
$D {\times} D$-stable basis. The proof
of Proposition \ref{1.7} is complete.

To fulfill the obligations taken on in Section
\ref{1}, it remains only to prove Theorem
\ref{1.4}. Some preparation is needed.
Recall, $b$ is said to be of {\bf principal
type} provided $\br_P(b)$ is a block of
$\FF C_G(P)$ for all $P \leq D$. That
is equivalent to the condition that the
restriction ${}_D \Res {}_G (\cO G b)$
is an almost-source $D$-algebra of
$\cO G b$. Corollary \ref{2.5} immediately
yields the next result.

\begin{pro}
\label{8.2}
If $b$ is of principal type then the
interior $D$-algebra ${}_D \Res {}_G
(\cO G b)$ is a uniform almost-source
$D$-algebra of $\cO G b$.
\end{pro}

The following theorem is part of
Harris--Linckelmann \cite[3.1]{HL00}.

\begin{thm}
\label{8.3}
{\rm (Harris--Linckelman.)}
Suppose $G$ is $p$-solvable. Then
there exists a subgroup $H \leq G$
and a block $c$ of $\cO H$ such that:

\noin {\bf (1)} $D$ is a defect group of
$c$ and a Sylow $p$-subgroup of $H$.

\noin {\bf (2)} $c$ is of principal type.

\noin {\bf (3)} We have $\cO G b =
\bigoplus_{fH, gH \subseteq G}
f \cO H c g^{-1}$ as $\cO$-modules,
$\cO G b \cong {}_G \Ind {}_H (\cO H c)$
as interior $G$-algebras and
$b = \tr_F^G(c)$.
\end{thm}

We shall also need the following
sufficient criterion for an idempotent
of $(\cO G b)^D$ to be an
almost-source idempotent.

\begin{lem}
\label{8.4}
Let $b$ be a block of $\cO G$ with
defect group $D$. Let $c$ be an
idempotent of $(\cO G b)^D$ such
that $\br_D(c) \neq 0$. Define
$A = c \cO G c$ as an interior
$D$-subalgebra of $\cO G b$. Suppose
that, for all $P \leq D$, the Brauer
quotient $A(P) \cong \br_P(c)
\FF C_G(P) \br_P(c)$ has a unique
block. Then $A$ is an almost source
algebra of $\cO G b$.
\end{lem}

\begin{proof}
Since $A$ is a permutation
$D$-algebra and $\br_D(c) \neq 0$,
each $\br_P(c) \neq 0$. Let $b_P$
be a block of $\FF C_G(P)$ such
that $\br_P(c) b_P \neq 0$. Then
$\br_P(c) b_P$ is a block of
$\br_P(c) \FF C_G(P) \br_P(c)$.
By the uniqueness hypothesis,
$\br_P(c) \leq b_P$.
\end{proof}

We now prove Theorem \ref{1.4}.
Let $H$ and $c$ be as in
Theorem \ref{8.3}. Put $A = c \cO G c$
as an interior $D$-algebra. By part (3)
of Theorem \ref{8.3}, $A =\cO H c$.
By the other two parts of that theorem,
together with Proposition \ref{8.1}, $A$
is a uniform almost-source algebra
of $\cO H c$. Given $P \leq D$, then
$A(P)$ is isomorphic to the block
algebra $\FF C_G(P) \br_P(c)$. Hence,
via Lemma \ref{8.4}, $A$ is an
almost-source $D$-algebra of
$\cO G b$. The proof of
Theorem \ref{1.4} is complete.


\begin{thebibliography}{EMG}

\bibitem[1]{Cra11}
D.\ A.\ Craven, {\it The Theory of Fusion Systems},
(Cambridge University Press, 2011).

\bibitem[2]{Gel19}
M.\ Gelvin, An observation on the module
structure of block algebras, {\it Comm.\ 
Algebra} {\bf 47}, (2019), 5286-5293.

\bibitem[3]{GR15}
M.\ Gelvin and S.-P.\ Reeh, Minimal
characteristic bisets for fusion systems,
{\it J.\ Algebra}, {\bf 427} (2015), 345--374.

\bibitem[4]{HL00}
M.\ E.\ Harris, M.\ Linckelmann, Splendid
derived equivalences for blocks of finite
$p$-solvable groups, {\it J.\ London
Math.\ Soc.\ } (Ser.\ 2) {\bf 62} (2000), 85-96.

\bibitem[5]{Lin18}
M.\ Linckelmann, {\it The Block Theory of
Finite Group Algebras}, Vols. 1, 2,
(Cambridge University Press, 2018).

\bibitem[6]{Pui86}
L.\ Puig, Local fusions in block source algebras,
{\it J.\ Algebra} {\bf 104} (1986), 358--369.

\bibitem[7]{PZ07}
L.\ Puig, Y.\ Zhou, A local property of basic
Morita equivalences, {\it Math.\ Zeit.\ }
{\bf 256} (2007), 551-562.

\end{thebibliography}
\end{document}